\documentclass{amsart}
\usepackage[utf8]{inputenc}
\usepackage{amsmath,amsfonts,amsthm,amssymb,hyperref}

\usepackage{xypic}

\usepackage[square,comma]{natbib}
\setcitestyle{numbers}

\newcommand{\tr}{\text{tr\,}}

\newcommand{\id}{\text{id}}

\newcommand{\Cmax}{C^\ast_\text{max}}

\newcommand{\N}{\mathbb{N}}

\newcommand{\R}{\mathbb{R}}
\newcommand{\C}{\mathbb{C}}

\newcommand{\calM}{\mathcal{M}}

\newcommand{\calS}{\mathcal{S}}

\newcommand{\ran}{\text{ran}}

\newcommand{\spn}{\text{span }}

\newcommand{\sa}{\text{sa}}

\newcommand{\epi}{\text{epi}}
\newcommand{\bary}{\text{bar}}

\newcommand{\Prob}{\text{Prob}}
\newcommand{\MIN}{\text{MIN}}
\newcommand{\CB}{\text{CB}}
\newcommand{\UCP}{\text{UCP}}

\newtheorem{thm}{Theorem}[section]
\newtheorem{cor}[thm]{Corollary}
\newtheorem{prop}[thm]{Proposition}

\newtheorem{ques}[thm]{Question}

\theoremstyle{definition}
\newtheorem{rem}[thm]{Remark}

\newtheorem{defn}[thm]{Definition}

\newtheorem{eg}[thm]{Example}

\title[Jensen's inequality for sep. nc convex functions]{Jensen's inequality for separately convex noncommutative functions}
\author{Adam Humeniuk}
\email{adam.humeniuk@uwaterloo.ca}
\date{\today}

\begin{document}

\begin{abstract}
Classically, Jensen's Inequality asserts that if $X$ is a compact convex set, and $f:K\to \mathbb{R}$ is a convex function, then for any probability measure $\mu$ on $K$, that $f(\text{bar}(\mu))\le \int f\;d\mu$, where $\text{bar}(\mu)$ is the barycenter of $\mu$. Recently, Davidson and Kennedy proved a noncommutative (``nc") version of Jensen's inequality that applies to nc convex functions, which take matrix values, with probability measures replaced by ucp maps. In the classical case, if $f$ is only a \emph{separately} convex function, then $f$ still satisfies the Jensen inequality for any probability measure which is a \emph{product} measure. We prove a noncommutative Jensen inequality for functions which are separately nc convex in each variable. The inequality holds for a large class of ucp maps which satisfy a noncommutative analogue of Fubini's theorem. This class of ucp maps includes any free product of ucp maps built from Boca's theorem, or any ucp map which is conditionally free in the free-probabilistic sense of M\l{}otkowski. As an application to free probability, we obtain some operator inequalities for conditionally free ucp maps applied to free semicircular families.
\end{abstract}

\subjclass[2010]{46A55, 46L07, 47A20}
\keywords{noncommutative functions, noncommutative convexity, Jensen inequality, free products, completely positive maps}

\maketitle


\section{Introduction}\label{sec:intro}

Noncommutative convexity is now an exciting and devloping toolbox for use in operator algebras and functional analysis. Wittstock \cite{wittstock_matrix_1984} introduced the central notion of a matrix convex set. The main idea is that matrix convex sets are graded by matrix levels, and include points at each level. Here, the classical notion of ``convex combination" $\sum_i t_ix_i$ is replaced with a ``matrix convex combination" $\sum_i \alpha_i^\ast x_i\alpha_i$, where the ``points" $x_i$ are matrices of possibly different sizes, and $\alpha_i$ are rectangular matrices satisfying $\sum_i \alpha_i^\ast\alpha_i = I$. Matrix convexity is a more natural notion for the study of operator algebras, where the study of structure at all matrix levels via completely positive or completely bounded maps is a central part of the theory. In fact, Webster and Winkler \cite{webster_krein-milman_1999} showed that the category of compact matrix convex sets is contravariantly equivalent to the category of operator systems, so matrix convex sets faithfully encode the information of any operator system.

The theory of matrix convex sets contains many noncommutative analogues of classical facts in convexity and Choquet theory. For instance, Effros and Winkler \cite{effros_matrix_1997} gave noncommutative analogues of the Hahn-Banach Separation Theorem and Bipolar Theorem. A persistent difficulty in matrix convexity was the search for the right notion of extreme point and a working Krein-Milman theorem. The best known version of a Krein-Milman type theorem in matrix convexity was given by Webster and Winkler in \cite{webster_krein-milman_1999}. Recently, Davidson and Kennedy \cite{davidson_noncommutative_2019} improved Webster and Winkler's result by working in a framework of noncommutative--or ``nc", convex sets, obtaining a Krein-Milman theorem and even a noncommutative Choquet-Bishop-De Leeuw Integral Representation Theorem.

The key idea in Davidson and Kennedy's framework is that one needs to include infinite matrix levels, and we refer to such convex sets as ``noncommutative" or ``nc" convex sets as opposed to ``matrix" convex sets. Noncommutative convex sets are determined by their finite levels (see \cite[Proposition 2.2.10]{davidson_noncommutative_2019}), so in a sense the two theories are the same. However, nc extreme points may occur at infinite matrix levels, and there are simple examples where the \emph{only} extreme points occur as infinite matrices. Using the dual equivalence to the category of operator systems, the extreme points in the state space of an operator system $S$ are exactly the (restrictions of) boundary representations of $S$, which completely norm the C*-envelope of $S$, in complete analogue to the Choquet boundary for a classical function system.

On matrix or nc convex sets, classical functions are more naturally replaced by noncommutative functions. Usually, one requires an nc function to be graded along matrix levels, to preserve direct sums, and respect similarities either by arbitrary invertible matrices, or just unitary equivalences. The theory of similarity invariant functions parallels complex analysis, because similarity invariant nc functions turn out to be automatically analytic. See \cite{kaliuzhnyi-verbovetskyi_foundations_2014} for a detailed treatment. Studying the notion of convex nc functions requires selfadjoint-valued functions, so that there is an ordering on the codomain. Because similarities don't preserve selfadjointness, we instead only require our nc functions in this context to be unitarily invariant.

If $X$ is a (classical) compact convex set, any convex function $f:X\to \R$ satisfies \emph{Jensen's inequality}
\[
f(\bary(\mu))\le 
\int_X f\;d\mu
\]
for any probability measure $\mu\in \Prob(X)$. The \emph{barycenter} $\bary(\mu)$ is the unique point in $X$ that satisfies $\varphi(\bary(\mu))=\int \varphi \;d\mu$ for every affine function $\varphi:X\to \C$. In fact, Jensen's inequality characterizes convexity, because we can take $\mu$ to be a convex combination of point masses.

In \cite[Section 7]{davidson_noncommutative_2019}, Davidson and Kennedy show that a noncommutative convex function
\[
f:K=\bigcup_n K(n)\to \calM=\bigcup_n M_n(\C)
\]
satisfies the Jensen inequality
\[
f(\bary(\mu))\le \mu(f)
\]
whenever $\mu$ is a ucp map $C(K)\to M_k$ defined on the C*-algebra $C(K)$ of continuous nc functions on $K$. Here, the \emph{barycenter} $\bary(\mu)$ of $\mu$ is the unique point in the $k$th matrix level $K(k)$ of $K$ that satisfies $a(\bary(\mu))=\mu(a)$ for all nc affine functions $a$ on $K$.

This noncommutative Jensen Inequality sheds some light on classical operator convexity. A function $f:I\to \mathbb{R}$ defined on some interval $I\subseteq \R$ is operator convex if its associated functional calculus defines a convex function, i.e. if
\[
f((1-t)x+ty)\le
(1-t)f(x)+tf(y)
\]
for all $t\in[0,1]$ and all \emph{selfadjoint matrices} $x,y$ with spectrum in $I$. Hansen and Pedersen \cite{hansen_jensens_2003-1} demonstrated that operator convexity is equivalent to a noncommutative Jensen inequality. Hansen and Pedersen's characterization can be obtained as a special case of Davidson and Kennedy's nc Jensen Inequality, by restricting to the case where
\[
K=\MIN(I) =\{x\in M_n^\sa\mid \sigma(x)\subseteq I\}
\]
is the unique minimal compact nc convex set with first level $\MIN(I)(1)=I$, see \cite[Section 4]{davidson_dilations_2018}. Operator convexity for multivariate functions is more delicate. For instance, Hansen \cite{hansen_jensens_2003} established a multivariate nc Jensen inequality for an operator convex function $f$ of two variables, but this inequality is more technical and can't be simply obtained from Davidson and Kennedy's inequality by working in $\MIN(I\times J)$, because on this domain an nc function may not be determined entirely by its first level.

\subsection{Main Results} Our main result is a fully noncommutative analogue of the following classical fact (see Proposition \ref{prop:classical_separate_jensen}). Let $X_1,\ldots,X_d$ be compact convex sets, and let $f:X_1\times\cdots \times X_d\to \R$ be a \emph{separately convex} function, meaning $f$ is convex as a function in any one variable as long as all other variables are fixed. This is a much weaker assumption than convexity of $f$. Then the function $f$ satisfies the Jensen inequality
\[
f(\bary(\mu))\le \int_{X_1\times\cdots \times X_d} f\; d\mu
\]
for any \emph{product} measure of the form $\mu=\mu_1\times\cdots \times \mu_d$, where $\mu_i\in \Prob(X_i)$. In fact, this characterizes separate convexity of $f$.

We are interested in convex nc functions of multiple variables. Therefore, in Section \ref{sec:products} we first study nc convex sets of the form $K_1\times\cdots \times K_d$, where $K_i$ are compact nc convex sets, and the product is taken separately at each matrix level. By the categorical duality between compact nc convex sets, each compact nc convex set $K_i$ corresponds to the operator system $A(K_i)$ of continuous nc affine functions on $K_i$. Because $K_1\times\cdots\times K_d$ is the categorical product, it follows from the equivalence of categories that (Proposition \ref{prop:A_product_coproduct})
\[
A(K_1\times\cdots \times K_d)\cong
A(K_1)\oplus_1\cdots \oplus_1 A(K_d)
\]
is the categorical coproduct of the associated operator systems. This is the ``unital direct sum" 
\[
S\oplus_1 T = \frac{S\oplus T}{\C((1_S,0)-(0,1_T))}
\]
constructed by Fritz in \cite{fritz_operator_2014}.

Davidson and Kennedy \cite[Section 4.4]{davidson_noncommutative_2019} showed that if $K$ is a compact nc convex set, the operator system $A(K)$ of continuous affine functions generates the C*-algebra $C(K)$ of (point-ultrastrong-$\ast$) continuous nc functions on $K$, and that
\[
C(K)=
\Cmax(A(K))
\]
is in fact the maximal C*-algebra. By comparing the right universal properties, it follows (Corollary \ref{cor:C_product}) that
\[
C(K_1\times\cdots \times K_d)\cong
C(K_1)\ast\cdots \ast C(K_d)
\]
is a free product of the associated maximal C*-algebras, with amalgamation over $\C$. This is evidently a noncommutative analogue of the classical result that $C(X_1\times\cdots\times X_d)\cong C(X_1)\otimes\cdots \otimes C(X_d)$ for compact Hausdorff spaces $X_1,\ldots,X_d$.

So, in analogy to the classical case, we should expect that a selfadjoint nc function $f:K_1\times\cdots \times K_d\to \calM^\sa$ which is \emph{separately nc convex} (Definition \ref{def:separately_nc_convex}) should satisfy an nc Jensen inequality for any ucp map
\[
\mu:C(K_1\times\cdots \times K_d)\cong
C(K_1)\ast\cdots \ast C(K_d)\to M_k
\]
which is a ``free product" of ucp maps $\mu_i:C(K_i)\to M_k$. The central difficulty is that the notion of ``free product" for ucp maps is not uniquely defined. If $A_i$, $i\in I$, are unital C*-algebras, then Boca's theorem \cite{boca_free_1991}, or its generalized version in \cite{davidson_proof_2019}, gives a standard recipe for how to glue a collection of ucp maps $\mu_i:A_i\to M_k$ to a ucp map
\[
\mu:\ast_{i\in I}A_i\to M_k
\]
with $\mu\vert_{A_i}=\mu_i$. However, such a map is not unique, and many such gluings might exist. Nonetheless, we show $f$ satisfies an nc Jensen inequality for \emph{any} ucp map built from Boca's theorem. We call any ucp map glued together as in the proof of Boca's theorem or the more general construction in \cite[Theorem 3.1]{davidson_proof_2019} a \emph{free product ucp map} (see Definition \ref{def:ucp_free_product}), and get the following result.

\begin{thm}\label{thm:main_separate_jensen}
Let $K_1,\ldots,K_d$ be compact nc convex sets, and suppose
\[
f:K_1\times\cdots\times K_d\to\calM^\sa
\]
is a continuous separately nc convex nc function. Suppose
\[
\mu:C(K_1\times\cdots\times K_d)\cong C(K_1)\ast\cdots\ast C(K_d)\to M_k
\]
is a free product ucp map of any ucp maps $\mu_i:C(K_i)\to M_k$, $i=1,\ldots,d$. Then $f$ satisfies the Jensen inequality
\[
f(\bary(\mu)) \le
\mu(f).
\]
\end{thm}

In fact, free product ucp maps are not the most general class of ucp maps on $C(K_1)\ast\cdots\ast C(K_d)$ for which we get a Jensen inequality. Our strongest version of Theorem \ref{thm:main_separate_jensen} is Theorem \ref{thm:separate_nc_jensen}, which shows that $f$ satisfies the Jensen inequality for any ucp map which satisfies a certain dilation-theoretic analogue of Fubini's theorem. We call such ucp maps ``Fubini type" (Definition \ref{def:fubini_type}), and they form a larger family than just maps coming from Boca's theorem. This class is large enough that the Jensen inequality characterizes separate nc convexity of $f$.

\subsection{Connection to Free Probability}

In \cite{bozejko_convolution_1996}, Bo\.{z}ejko, Leinert, and Speicher introduced the notion of a \emph{conditionally free} or \emph{c-free} product of states on a free product $A_1\ast\cdots\ast A_d$ of C*-algebras. M\l{}otowski \cite{mlotkowski_operator-valued_2002} generalized this definition to include a conditionally free product of ucp maps as follows. Suppose we have an index set $I$, unital C*-algebras $A_i$, $i\in I$, and prescribed ucp maps
\begin{align*}
\mu_i:A_i&\to M_k, \quad\text{ and states }\\
\varphi_i:A_i&\to \C,
\end{align*}
for $i\in I$. The \emph{$(\varphi_i)_{i\in I}$-conditionally free product} of the ucp maps $\mu_i$ is a ucp map
\[
\mu:\ast_{i\in I}A_i\to M_k
\]
which satisfies $\mu\vert_{A_i}=\mu_i$ for each $i\in I$, and whenever $a_1\cdots a_m\in \ast_{i\in I}A_i$ is a reduced word (meaning $a_\ell\in A_{j_\ell}$ with $j_1\ne j_2\ne\cdots\ne j_m$) that satisfies
\[
\varphi_{j_\ell}(a_\ell)=0
\]
for each $\ell=1,\ldots,m$, then one has the independence rule
\[
\mu(a_1\cdots a_m)=0.
\]
If $\mu$ is a conditionally free product for any tuple of states $(\varphi_i)_{i\in I}$, we will simply say $\mu$ is a conditionally free or c-free ucp map. We can decompose each $A_i$ as a direct sum $A_i=\ker\varphi_i\oplus \C 1_{A_i}$, the value of $\mu$ on any reduced word is recursively determined by $\mu_i$ and $\varphi_i$ for $i\in I$. Thus the $(\varphi_i)_{i\in I}$-c-free product is uniquely determined.

Existence and complete positivity of the $(\varphi_i)_{i\in I}$-c-free product $\mu$ follows from Boca's theorem. Indeed, examining \cite[Theorem 3.1]{boca_free_1991} or \cite[Theorem 3.4]{davidson_proof_2019} in the case of amalgamation over $\C$ shows that the constructed ucp map on the free product is the unique c-free product ucp map on the unital free product. Since c-free products can be built from Boca's theorem, they are product ucp maps in our definition and so Theorem \ref{thm:main_separate_jensen} in this context gives

\begin{cor}\label{cor:main_c_free}
Suppose $K_1,\ldots,K_d$ are compact nc convex sets and let $A_i=C(K_i)$, $i=1,\ldots,d$. Then for any continuous separately nc convex function
\[
f:K_1\times\cdots \times K_d\to \calM^\sa,
\]
and any conditional free ucp map
\[
\mu:A_1\ast\cdots \ast A_d\to M_k,
\]
the Jensen inequality
\[
f(\bary(\mu))\le
\mu(f)
\]
holds.
\end{cor}

Note that Corollary \ref{cor:main_c_free} applies exactly to those unital C*-algebras $A_i$ which are of the form $A_i=\Cmax(S_i)$ for any operator systems $S_i$. In this case, we may assume $K_i=\calS(S_i)$ is the nc state space $\bigcup_n \UCP(S_i,M_n)$. For example, the result applies to commutative C*-algebras of the form $C(X_i)$, where $X_i\subseteq\mathbb{R}^m$ are simplices.

As an application of Corollary \ref{cor:main_c_free}, we obtain some operator inequalities for conditionally free ucp maps on free semicircular families. For instance, let $a$ and $b$ be free semicircular elements in a C*-probability space $(A,\varphi)$, where $\varphi$ is faithful and tracial. Let $S=\spn\{1_A,a\}$ and $T=\spn\{1_A,b\}$ be the operator systems they generate Then because the spectra $\sigma(a)$ and $\sigma(b)$ are closed intervals, the continuous functional calculus implies that
\[
C^\ast(a)\cong 
\Cmax(S) \quad\text{and}\quad
C^\ast(b)\cong
\Cmax(T).
\]
Therefore Corollary \ref{cor:main_c_free} applies to 
\[
C^\ast(a,b)\cong
C^\ast(a)\ast C^\ast(b).
\]
With this identification, elements such as $ab+ba$ or $ab^2a$ correspond to separately nc convex functions. Consequently, if $\mu:C^\ast(a)\ast C^\ast(b)\to B(H)$ is a c-free ucp map, or a ucp map built from Boca's theorem, in Example \ref{eg:semicircular_jensen} we obtain the operator inequalities
\begin{align*}
\mu(a)\mu(b)+\mu(b)\mu(a)&\le \mu(ab+ba)\quad\text{and}\\
\mu(a)\mu(b)^2\mu(a)&\le \mu(ab^2a).
\end{align*}

More generally, if $a_1,\ldots,a_k$ is any free semicircular family in $(A,\varphi)$, and $\mu$ is a c-free ucp map, we show (Corollary \ref{cor:semicircular_inequalities}) that
\begin{align*}
\mu(a_1)\cdots \mu(a_k)+\mu(a_k)\cdots \mu(a_1) &\le
\mu(a_1\cdots a_k+a_k\cdots a_1),\quad \text{and} \\
\mu(a_1)\cdots \mu(a_{k-1})\mu(a_k)^2\mu(a_{k-1})\cdots \mu(a_1)&\le
\mu(a_1\cdots a_{k-1}a_k^2a_{k-1}\cdots a_1).
\end{align*}

The appearance of conditional freeness in Corollary \ref{cor:main_c_free} suggests that some analogue of free independence for ucp maps may play a role in our main Theorem \ref{thm:separate_nc_jensen}.

\begin{ques}\label{ques:independence}
Is the class of ucp maps $C(K_1)\ast\cdots \ast C(K_d)\to M_k$ for which a noncommutative Jensen inequality for separately nc convex functions holds described by some free independence condition? In the language of Section \ref{subsec:nc_jensen}, are ucp maps of Fubini type, or free product ucp maps, characterized by some generalized free independence condition?
\end{ques}

Question \ref{ques:independence} has a positive answer for states, in which case $k=1$ and $M_k=\C$. Indeed it is straightforward to check that if
\[
\varphi:\ast_{i\in I}A_i\to \C
\]
is a state which is a free product ucp map (Definition \ref{def:ucp_free_product}), then the C*-subalgebras $A_i$ are freely independent, and this occurs if and only if $\varphi$ is the unique free product of the states $\varphi_i:=\varphi\vert_{A_i}$ built from Boca's theorem or by \cite[Proposition 1.1]{avitzour_free_1982}.

\section{Background}\label{sec:background}

\subsection{Noncommutative Convexity} Throughout, we work in the framework of nc convexity developed by Davidson and Kennedy in \cite{davidson_noncommutative_2019}. Because their results are still fairly novel, we devote a larger-than-normal portion of this section to an exposition of the main results of their paper that play a role here.

Given an operator system $E$, we let
\[
\calM(E) =
\coprod_{n\le \kappa} M_n(E),
\]
where $\kappa$ is any fixed sufficiently large cardinal greater than the density character of $E$. In the special case where $E$ is separable, usually $\kappa=\aleph_0$. When $E=\C$, we write
\[
\calM:=
\calM(E)
\]
for simplicity. The key difference from the existing theory of matrix convex sets is that we allow for infinite matrix levels, i.e. we consider all cardinals $n\le \alpha$, with the convention that $M_n(E)\cong M_n\otimes_{\text{min}}E$. Note that when $E=\C$, by convention we have $M_n:=M_n(\C)=B(H_n)$ for any Hilbert space $H_n$ of dimension $n$.

If $E$ is an operator system, we call a subset 
\[
K\subseteq \calM(E)
\]
an \emph{nc convex set} if it is closed under direct sums and compression by isometries. Equivalently, given $x_i\in K(n_i)$ for some index set $i\in I$, and matrices $\alpha_i\in M_{n_i,n}$ satisfying
\[
\sum_{i\in I} \alpha_i^\ast \alpha_i = I_n,
\]
where the series converges weak-$\ast$, the ``nc convex combination"
\[
\sum_{i\in I} \alpha_i^\ast x_i \alpha_i
\]
is also in $K$. The $n$th \emph{matrix level} of $K$ is
\[
K(n):=
K\cap M_n(E),
\]
and if $K$ is nonempty and nc convex, then each $K(n)$ is nonempty. If $E=(E_\ast)^\ast$ is a dual operator system, then we may identify
\[
M_n(E) =
M_n(\CB(E_\ast,\C)) \cong
\CB(E_\ast,M_n)\subseteq
B(E_\ast,M_n),
\]
which has a standard weak-$\ast$ topology. Hence $M_n(E)$ has an induced weak-$\ast$ topology, with agrees with the topology of pointwise-weak-$\ast$ convergence on $B(E_\ast,M_n)$ on bounded subsets. When $E=\C$, this is just the usual weak-$\ast$ topology on each level $\calM(n)=M_n=B(H_n)$ of $\calM$. If $E$ is a dual operator space, we say that $K$ is a \emph{compact nc convex set} if it is nc convex and each level $K(n)\subseteq M_n(E)$ is compact in the weak-$\ast$ topology.

Given nc convex sets $K$ and $L$, a function $f:K\to L$ is called an \emph{nc function} if $f$ is graded (i.e. $f(K(n))\subseteq L(n)$), preserves unitary equivalence, and direct sums. If $K$ and $L$ are compact nc convex sets, we say $f$ is continuous if it is weak-$\ast$ continuous on each level of $K$. Moreover $f$ is \emph{nc affine} if it also preserves compressions, so if $x\in K(n)$, and $\alpha \in M_{n,k}$ is an isometry, then
\[
f(\alpha^\ast x\alpha) =
\alpha^\ast f(x)\alpha.
\]

Compact nc convex sets are a noncommutative analogue of classical compact convex sets (in locally convex spaces). In the classical case, given a compact convex set $X$, the space $A(X)$ of continuous affine function $X\to \C$ forms a function system. Kadison's Representation Theorem \cite{kadison_representation_1951} shows that the functor $A:x\mapsto A(X)$ is an equivalence of categories between the category of compact convex sets and the category of function systems. The essential inverse functor is $F\mapsto \calS(F)$, where $\calS(F)$ is the state space of $F$, and so the map $C\to \calS(A(C))$ which embeds $C$ as point evaluations is a natural isomorphism.

In the noncommutative setting, operator systems are the correct analogue of function systems. Given a compact nc convex sets $K$ and $L$, we form the space
\[
A(K,L)=
\{a:K\to L\mid a\text{ nc affine}\}
\]
of nc affine functions with values in $L=\bigcup_n L_n$. When $L=\calM$, we set $A(K):=A(K,\calM)$. The space $A(K)$ is an operator system, with $\ast$-structure
\[
a^\ast(x):=
a(x)^\ast.
\]
The matrix order unit is the ``constant function"
\[
1_{A(K)}(x)=1_{M_n}, \quad x\in K(n).
\]
The matrix order structure on $M_n(A(K))\cong A(K,M_n\calM)$ is pointwise. The functor $K\mapsto A(K)$ implements an equivalence of categories between the category of compact nc convex sets, with continuous nc affine maps as morphisms, and the category of operator systems, with ucp maps as morphisms \cite[Theorem 3.2.5]{davidson_noncommutative_2019}. The essential inverse takes an operator system $S$ to the nc state space
\[
\calS(S) \subseteq
\calM(S^\ast)
\]
whose $n$th level is
\[
\calS(S)(n)=
\{\mu:S\to M_n\mid \mu \text{ ucp}\}.
\]
In particular $\calS(A(K))\cong K$ naturally, and the isomorphism means that every ucp map $A(K)\to M_n$ is of the form $f\mapsto f(x)$, for some $x\in K(n)$.

Given an operator system $S$, the maximal C*-algebra $\Cmax(S)$ satisfies the following universal property. We have an embedding
\[
S\subseteq \Cmax(S)=C^\ast(S),
\]
and for any unital complete order embedding $\iota:S\to B(H)$, there is a unique $\ast$-homomorphism $\pi:\Cmax(S)\to B(H)$ with $\pi\vert_S=\iota$. In the classical setting, for a compact convex set $X$, one has $\Cmax(A(X))\cong C(X)$, via the usual inclusion $A(X)\subseteq C(X)$. Let $K$ be a compact nc convex set. Let $B(K)$ denote the C*-algebra of bounded nc functions $K\to \calM$, where the C*-operations are pointwise. Then $A(K)\subseteq B(K)$, and we set
\[
C(K):=C^\ast(A(K))\subseteq B(K).
\]
Davidson and Kennedy demonstrated a noncommutative analogue of the classical result $\Cmax(A(X))\cong C(X)$ in \cite[Theorem 4.4.3]{davidson_noncommutative_2019}. The C*-algebra $C(A(K))$ is both
\begin{itemize}
\item the C*-algebra of all bounded nc functions $K\to \calM$ which are continuous levelwise in the point-ultrastrong-$\ast$ topology on each $K(n)\subseteq B(E_\ast,M_n)$, and
\item the maximal C*-algebra $\Cmax(A(K))$, with the usual inclusion $A(K)\subseteq C(K)$.
\end{itemize}
By the universal property and category duality, any $\ast$-homomorphism $\pi:C(K)\to M_n$ is of the form $\pi=\delta_x:f\mapsto f(x)$, for some $x\in K(n)$. Then, Stinespring's theorem implies that any ucp map $\mu:C(K)\to M_n$ has the form 
\[
\mu=\alpha^\ast \delta_x\alpha:f\mapsto \alpha^\ast f(x)\alpha,
\]
for some $x\in K_k$ and isometry $\alpha \in M_{k,n}$.

On a classical compact convex set $X$, a function $f:X\to \R$ is \emph{convex} if for all $x_1,\ldots,x_n\in X$ and $t_1,\ldots,t_n\in [0,1]$ with $\sum_{k=1}^n t_k = 1$, we have
\[
f\Big(\sum_{k=1}^n t_kx_k\Big)\le
\sum_{k=1}^n t_kf(x_k).
\]
If $f$ is continuous, then $f$ is convex if and only if $f$ satisfies \emph{Jensen's inequality}
\[
f(\bary(\mu))\le \int f\; d\mu
\]
for all Radon probability measures $\mu\in \Prob(X)$. Here, the \emph{barycenter} $\bary(\mu)$ is the unique point in $X$ such that $a(\bary(\mu))=\int a\; d\mu$, which exists by Kadison duality.

In the noncommutative case, a selfadjoint nc function $f:K\to \calM^\sa$ on a compact nc convex set $K$ is \emph{nc convex} if whenever $x_i\in K(n_i)$ and $\alpha_i\in M_{n_i,n}$ with $\sum_i \alpha_i^\ast \alpha_i=I_n$, we have
\[
f\left(\sum_{i\in I} \alpha_i^\ast x_i\alpha_i\right)\le 
\sum_{i\in I} \alpha_i^\ast f(x_i)\alpha_i.
\]
$f$ is a continuous nc function, it automatically preserves direct sums, and so it is equivalent to simply require
\begin{equation}\label{eq:nc_convex}
f(\alpha^\ast x\alpha)\le
\alpha^\ast f(x)\alpha
\end{equation}
whenever $x\in K(k)$ and $\alpha\in M_{k,n}$ is an isometry. Consequently, a continuous bounded nc function $f\in C(K)^\sa$ is nc convex if and only if it satisfies the \emph{nc Jensen inequality}
\begin{equation}\label{eq:nc_jensen}
f(\bary(\mu))\le \mu(f)
\end{equation}
for all ucp maps $\mu:C(K)\to M_n$. Here the \emph{barycenter} of a ucp map $\mu:C(K)\to M_n$ is the unique point $\bary(\mu)\in K(n)$ such that $\mu(a)=a(\bary(\mu))$ for all $a\in A(K)\subseteq C(K)$. The nc Jensen inequality above follows directly from \eqref{eq:nc_convex} together with the observation that any ucp map $\mu:C(K)\to M_n$ must be of the form $\mu=\alpha^\ast \delta_x\alpha$ for a point $x\in K$ and isometry $\alpha$, and in this case $\bary(\mu)=\alpha^\ast x\alpha\in K(n)$. In fact  \cite[Theorem 7.6.1]{davidson_noncommutative_2019} applies even to matrix-valued bounded nc functions $f:K\to M_k(\calM)^\sa$ which are \emph{lower semicontinuous} in the sense that their \emph{nc epigraph}
\[
\epi(f) =
\bigcup_{n\le \kappa} \{(x,y)\in K(n)\times M_k(M_n)\mid y\ge f(x)\}
\]
is levelwise weak-$\ast$ closed, see \cite[Theorem 7.6.1]{davidson_noncommutative_2019}.

\subsection{Minimal nc convex sets} If $X\subseteq \C^d$ is a compact convex set, there is a minimal compact nc convex set $K=\MIN(X)\subseteq \calM^d$ with $K(1)=X$ \cite[Definition 4.1]{davidson_dilations_2018}. A tuple $x=(x_1,\ldots,x_d)\in \calM^d(n)=M_n^d$ lies in $\MIN(I)(n)$ if and only if there is a normal tuple $n=(n_1,\ldots,n_d)\in \calM^d$ which dilates $x$ and has joint spectrum
\[
\sigma(n)\subseteq X.
\]
In particular, if $d=1$ and $X=[a,b]$ is an interval, then
\[
\MIN([a,b])=
\{x\in \calM^\sa \mid \sigma(x)\subseteq [a,b]\}.
\]

For general $X\subseteq \C^d$, if $a\in A(K)$ with first level $f:=a\vert_{K(1)}=a\vert_X$, because $a$ preserves unitaries and direct sums, an application of the spectral theorem shows that 
\[
a(n_1,\ldots,n_d)=f(n_1,\ldots,n_d),
\]
in the sense of the functional calculus, for every normal tuple $n=(n_1,\ldots,n_d)$ with joint spectrum $\sigma(n)\subseteq X$. Because $a$ preserves compressions, and every $x\in \MIN(I)$ is a compression of a normal tuple, the nc affine function $a$ is determined by its restriction to the first level. Consequently $A(\MIN(X))\cong A(X)$ via the isomorphism that restricts to the first level.

\subsection{Dilations and Notation} If $E$ is an operator space, $x\in M_n(E)$, and $y\in M_k(E)$, we say that $y$ \emph{dilates} $x$ if there is an isometry $\alpha \in M_{k,n}$ such that $\alpha^\ast y \alpha=x$. In this case we write, $x\prec_\alpha y$, or just $x\prec y$ if the associated isometry is clear.

If $S$ is an operator system and $\mu:S\to B(H)$, $\nu:S\to B(K)$ are ucp maps, then we say $\nu$ dilates $\mu$ if $\alpha^\ast \nu\alpha=\mu$, for some isometry $\alpha:H\to K$. Usually up to a unitary we assume $\alpha$ is just an inclusion map $H\subseteq K$. This is really the same perspective as above, where $E=S^\ast$ is the dual operator space and there is a standard identification
\[
M_n(E)\cong
\CB(S,M_n),
\]
where for infinite cardinals $n$ we take $M_n=B(H_n)$ for an $n$-dimensional Hilbert space $H_n$, and identify
\[
H\cong
H_{\dim H} \quad \text{and}\quad
K\cong H_{\dim K}
\]
up to some fixed hidden unitary. Since our nc functions are always unitarily equivariant, these hidden unitaries are harmless, so we may freely switch between working with $M_n$ and $B(H)$ as long as $\dim H=n$.

A dilation $x\prec y$ is \emph{trivial} if $y\cong x\oplus z$ with respect to the range of $\alpha$, or equivalently $\alpha \alpha^\ast y=y\alpha\alpha^\ast$. For ucp maps $\mu:S\to B(H)$ and $\nu:S\to B(K)$, with $H\supseteq K$ and $\mu\prec \nu$, the dilation $\mu\prec \nu$ is trivial if and only if $H$ is invariant/reducing for $\nu(S)$.

\subsection{Free Products}

If $A_i$, $i\in I$ are unital C*-algebras, we denote their unital free product C*-algebra by
\[
\ast_{i\in I}A_i,
\]
or, if $I=\{1,\ldots,d\}$, by $A_1\ast\cdots\ast A_d$. Here, we amalgamate only over the subalgebras $\C\cong \C1_{A_i}\subseteq A_i$. For convenience, we freely identify each C*-algebra $A_j$ as a literal subalgebra $A_j\subseteq \ast_{i\in I} A_i$ of the free product.

\section{Products of NC Convex Sets}\label{sec:products}

Suppose $K_1\subseteq \calM(E_1)$ and $K_2\subseteq \calM(E_2)$ are compact nc convex sets, where $E_i=(E_{i,\ast})^\ast$ are dual operator systems. As in \cite{davidson_noncommutative_2019}, compactness is meant levelwise in the weak-$\ast$ topology. The Cartesian product
\[
K_1\times K_2 :=
\coprod_n (K_1)(n)\times (K_2)(n) \subseteq
\calM(E_1\times E_2)
\]
is also an nc convex set. By convention $E_1\times E_2$ is the usual $\ell^\infty$-product of operator spaces. We have the standard operator space duality \cite[Section 2.6]{pisier_introduction_2004}
\[
E_1\times E_2=((E_1)_\ast\times_1 (E_2)_\ast)^\ast
\]
and the corresponding weak-$\ast$ topology agrees with the product topology on $K_1\times K_2$. Hence $K_1\times K_2$ is a compact nc convex set when given the product topology. It is straightforward to verify that $E_1\times E_2$ is the categorical product of $E_1$ and $E_2$ in the category of compact nc convex sets with continuous nc affine maps as morphisms.

Davidson and Kennedy \cite[Theorem 3.2.5]{davidson_noncommutative_2019} showed that the functor $K\mapsto A(K)$ implements an equivalence of categories between this category of compact nc convex sets and the category of operator systems with ucp maps as morphisms. Fritz \cite[Proposition 3.3]{fritz_operator_2014} showed that the categorical coproduct in the category of operator systems $S,T$ is the \emph{unital direct sum}
\[
S\oplus_1 T :=
\frac{S\times T}{\C((1_S,0)-(0,1_T))}.
\]
Here, we naturally identify
\[
M_n(S\oplus_1 T)\cong
\frac{M_n(S)\times M_n(T)}{M_n(\C((1_S,0)-(0,1_T)))}.
\]
Write the coset of a pair $(s,t)\in M_n(S)\times M_n(T)$ in $M_n(S\oplus_1 T)$ as $s\oplus_1 t$. The matrix order structure is determined by declaring
\[
s\oplus_1 t\ge 0 \iff
s-\lambda \ge 0 \text{ and }t+\lambda \ge 0 \text{ for some }\lambda \in M_n(\C).
\]

\begin{prop}\label{prop:A_product_coproduct}
Let $K_1\subseteq \calM(E_1)$ and $K_2\subseteq \calM(E_2)$ be compact nc convex sets. Then there is a natural complete order isomorphism
\[
A(K_1\times K_2)\cong
A(K_1)\oplus_1 A(K_2).
\]
Here $a\oplus_1 b\in A(K_1)\oplus_1 A(K_2)$ corresponds to the continuous nc affine function
\[
(a\oplus_1 b)(x,y)=a(x)+b(y),\quad
(x,y)\in K_1\times K_2.
\]
\end{prop}

\begin{proof}
By \cite[Theorem 3.2.5]{davidson_noncommutative_2019}, the functor $K\mapsto A(K)$ is a contravariant equivalence of categories. Let $\pi_i:K_1\times K_2\to K_i$ be the usual projection. Since the diagram
\[\xymatrix @=1em{
	& K_1 \\
	K_1\times K_2 \ar[ur]^{\pi_1} \ar[dr]_{\pi_2}\\
	& K_2
}\]
is a categorical product in the category of compact nc convex sets, the diagram
\[
\xymatrix @=1em{
	A(K_1) \ar[dr]^{\epsilon_1} \\
	& A(K_1\times K_2) \\
	A(K_2) \ar[ur]_{\epsilon_2}
}\]
is a coproduct of operator systems, where $\epsilon_i(a)=a\circ \pi_i$. Since the coproduct $A(K_1)\oplus_1 A(K_2)$ is unique up to isomorphism, the induced map
\[
\epsilon_1\oplus_1\epsilon_2:A(K_1)\oplus_1 A(K_2)\to A(K_1\times K_2)
\]
given by $(\epsilon_1\oplus\epsilon_2)(a,b)=\epsilon_1(a)+\epsilon_2(b)$ is an isomorphism.
\end{proof}

For a compact nc convex set $K$, recall that $C(K)=\Cmax(A(K))$ is the maximal C*-algebra generated by the operator system $A(K)$. Since $A(K_1\times K_2)=A(K_1)\oplus_1 A(K_2)$ is a coproduct of operator systems, it is natural to expect that
\[
C(K_1\times K_2)\cong
\Cmax(A(K_1)\oplus_1 A(K_2))
\]
is itself a coproduct in the category of unital C*-algebras, with unital $\ast$-homomorphisms as morphisms. Indeed this is the case. Here, the coproduct of unital C*-algebras $A$ and $B$ is the unital free product $A\ast B$ with amalgamation over $\C\cong \C1_A\cong \C1_B$.

\begin{prop}\label{prop:Cmax_coproduct}
Let $S$ and $T$ be operator systems. Then
\[
\Cmax(S\oplus_1 T)\cong
\Cmax(S)\ast \Cmax(T)
\]
naturally. The $\ast$-isomorphism $\Cmax(S)\ast\Cmax(T)\to \Cmax(S\oplus_1 T)$ is induced by the $\ast$-monomorphisms
\[
\iota_S:\Cmax(S)\to \Cmax(S\oplus_1 T),\quad
\iota_T:\Cmax(T)\to \Cmax(S\oplus_1 T),
\]
which are themselves induced by the natural complete order embeddings $\iota_S:S\to S\oplus_1 T$ and $\iota_T:T\to S\oplus_1 T$.
\end{prop}

\begin{proof}
It suffices to show that
\[\xymatrix @=1em{
	\Cmax(S) \ar[dr]^{\iota_S} \\
   	& \Cmax(S\oplus_1 T) \\
	\Cmax(T) \ar[ur]_{\iota_T}
}\]
is a coproduct in the category of unital C*-algebras, and the natural isomorphism $\Cmax(S\oplus_1 T)\cong \Cmax(S)\ast\Cmax(T)$ follows by uniqueness of coproducts up to isomorphism. 

Suppose $A=B(H)$ is a C*-algebra and we have $\ast$-homomorphisms $\pi_S:\Cmax(S)\to A$ and $\pi_T:\Cmax(T)\to A$. Set $\varphi_S=\pi_S\vert_S$ and $\varphi_T=\pi_T\vert_T$, which are ucp maps $S\to A$ and $T\to A$, respectively. By the universal property, these induce a ucp map $\varphi:S\oplus_1 T\to A$ with $\varphi\iota_S=\varphi_S$ and $\varphi\iota_T=\varphi_T$. The ucp map $\varphi$ induces a $\ast$-homomorphism
\[
\pi:\Cmax(S\oplus_1 T)\to B(H)
\]
with $\pi\vert_{S\oplus_1 T}=\varphi$. Because $S\oplus_1 T$ generates $\Cmax(S\oplus_1 T)$, and $\varphi(S\oplus_1 T)\subseteq A$, we in fact have
\[
\pi(\Cmax(S\oplus_1 T))=
C^\ast(\varphi(S\oplus_1 T))\subseteq A.
\]
Since $\iota_S(\Cmax(S))=C^\ast(\iota_S(S))$ is generated by $\iota_S(S)=S\oplus_1 0$, and $\pi\iota_S\vert_S=\varphi\iota_S=\varphi_S=\pi_S\vert_S$, we have $\pi\iota_S=\pi_S$. Identically, we find $\pi\iota_T=\pi_T$. Since the $\ast$-homomorphism $\pi$ is determined by its action on the generating set $S\oplus_1 T=\iota_S(S)+\iota_T(T)$, it follows that $\pi$ is unique. 
\end{proof}

\begin{cor}\label{cor:C_product}
Let $K_1\subseteq \calM(E_1)$ and $K_2\subseteq \calM(E_2)$ be compact nc convex sets. Then
\[
C(K_1\times K_2)\cong
C(K_1)\ast C(K_2)
\]
naturally via the isomorphism $\epsilon_1\ast \epsilon_2:C(K_1)\ast C(K_2)$, where
\[
\epsilon_1(f)(x,y)=f(x) \text{ and }
\epsilon_2(g)(x,y)=g(y)
\]
for $f\in C(K_1)$ and $g\in C(K_2)$.
\end{cor}

\begin{proof}
This follows from combining Proposition \ref{prop:A_product_coproduct} with Proposition \ref{prop:Cmax_coproduct}. Note that each inclusion
\[
A(K_i)\to A(K_1)\oplus_1 A(K_2)\cong A(K_1\times K_2)
\]
is exactly the map $\epsilon_i$, which clearly extends to a $\ast$-homomorphism $C(K_i)\to C(K_1\times K_2)$. Thus when identifying $\Cmax(A(K_i))\cong C(K_i)$, in the notation of Proposition \ref{prop:Cmax_coproduct} we must have $\iota_{A(K_i)}=\epsilon_i$, so the isomorphism is implemented by 
\[
\iota_{A(K_1)}\ast \iota_{A(K_2)}=\epsilon_1\ast \epsilon_2.\qedhere
\]
\end{proof}

\begin{rem}\label{rem:many_variables}
Propositions \ref{prop:A_product_coproduct} and \ref{prop:Cmax_coproduct}, and Corollary \ref{cor:C_product} all extend immediately to finitely many variables. Since categorical products and coproducts such as $\times$, $\oplus_1$, and $\ast$ are all associative up to natural isomorphism, a straightforward induction shows that for any $d\in \N$, we have natural isomorphisms
\begin{align*}
A(K_1\times\cdots \times K_d), &\cong
A(K_1)\oplus_1\cdots \oplus_1 A(K_d) \\
\Cmax(S_1\oplus_1\cdots \oplus_1 S_d) &\cong
\Cmax(S_1)\ast\cdots \ast \Cmax(S_d), \text{ and} \\
C(K_1\times\cdots \times K_d) &\cong C(K_1)\ast \cdots \ast C(K_d),
\end{align*}
whenever $K_1,\ldots,K_d$ are compact nc convex sets and $S_1,\ldots,S_d$ are operator systems.
\end{rem}

\section{Jensen's Inequality for Separately NC Convex Functions}\label{sec:separate_jensen}

\subsection{The commutative case}\label{subsec:classical_separate_jensen}

Let $K_1,\ldots,K_d$ be (classical) compact convex sets. A function $f:K_1\times\cdots\times K_d\to \R$ is separately convex if it is convex in each variable separately. That is, if 
\[
f((1-t)x+ty) \le
(1-t)f(x)+tf(y)
\]
whenever the points  $x=(x_1,\ldots,x_d)$ and $y=(y_1,\ldots,y_d)$ in $K_1\times\cdots \times K_d$ differ in at most one coordinate. 

\begin{eg}\label{eg:classical_separately_convex}
The function $f:\R^2\to \R$ given by $f(x,y)=xy$ is separately convex but not convex. Indeed, it's affine in each variable, yet for example we find
\[
f(1/2,-1/2)=-1/4>-1/2=\frac{f(1,-1)}{2}+\frac{f(0,0)}{2}.
\]
\end{eg}

Separately convex functions satisfy Jensen's inequality for product measures.

\begin{prop} \label{prop:classical_separate_jensen}
Let $K_1,\ldots,K_d$ be compact convex sets, and let $f:K_1\times\cdots\times K_d\to \mathbb{R}$ be a continuous function. Then $f$ is separately convex if and only if
\[
f(\bary(\mu))\le \int_{K_1\times\cdots \times K_d} f\; d\mu
\]
for all product measures $\mu=\mu_1\times\cdots\times \mu_d$, where $\mu_k\in \Prob(K_k)$ for each $1\le k \le d$.
\end{prop}

\begin{proof}
For simplicity we will prove only the case $d=2$. Suppose $f$ satisfies $f(\bary(\mu))\le \mu(f)$ for any product of probability measures. Suppose $x,y\in K_1$, $t\in [0,1]$, and $z\in K_2$. Define the product measure
\[
\mu :=
\left((1-t)\delta_{x}+t\delta_{y}\right)\times \delta_z
\]
Then
\begin{align*}
f((1-t)x+ty,z)&=
f(\bary(\mu)) \le
\mu(f)\\&=
(1-t)f(x,z)+tf(y,z).
\end{align*}
Therefore the function $f$ is convex in its first argument, and a symmetrical argument works in the second argument.

Conversely, suppose $f$ is separately convex, and let $\mu=\mu_1\times \mu_2$ be a product of probability measures. By Fubini's theorem, we find
\begin{align*}
\mu(f) &=
\int_{K_1}\int_{K_2} f(x,y)\; d\mu_2(y)\; d\mu_1(x) \\ &\ge
\int_{K_1} f(x,\bary(\mu_2)) \;d\mu_1(x)\\ &\ge
f(\bary(\mu_1),\bary(\mu_2)) =f(\bary(\mu)).
\end{align*}
Here, in the first inequality, we had
\[
f(x,\bary(\mu_2))\le \int_{K_1} f(x,y)\; d\mu_2(y)
\]
by the one-variable version of Jensen's inequality applied to the convex function $y\mapsto f(x,y)$, and similarly in the second.
\end{proof}

Taking a more algebraic perspective, we have the standard identification
\[
C(K_1\times\cdots \times K_d)\cong
C(K_1)\otimes\cdots \otimes C(K_d)
\]
as C*-algebras. Product measures, as states on $C(K_1\times\cdots \times K_d)$, are exactly those states of the form $\mu_1\otimes\cdots\otimes \mu_d$, where each $\mu_i\in S(C(K_i))$ is a state.

\subsection{Noncommutative analogue}\label{subsec:nc_jensen}

In the setting of noncommutative convexity, we've seen in Corollary \ref{cor:C_product} that
\[
C(K_1\times\cdots \times K_d) \cong
C(K_1)\ast\cdots \ast C(K_d)
\]
for compact nc convex sets $K_1,\ldots,K_d$. The expected noncommutative analogue of Proposition \ref{prop:classical_separate_jensen} should be that a ``separately nc convex" nc function $f:K_1\times \cdots \times K_d\to \calM^\sa$ satisfies a Jensen-type inequality for any nc state/ucp map
\[
\mu:C(K_1\times \cdots \times K_d)\cong C(K_1)\ast\cdots \ast C(K_d)\to M_m
\]
that arises as a ``free product" of ucp maps $\mu_i:C(K_i)\to M_m$. Defining a noncommutative version of separate convexity is straightforward, but the main difficulty is identifying which ucp maps should ``count" as free products. Boca's theorem \cite{boca_free_1991} and its more general version given by Davidson and Kakariadis \cite{davidson_proof_2019} provide a standard recipe for taking free products for ucp maps. We will show ucp maps built from Boca's theorem satisfy a Jensen inequality for separately nc convex functions, but this is not the biggest class that works.

\begin{defn}\label{def:separately_nc_convex}
Suppose $K_i$, $i\in I$, are compact nc convex sets, for $i\in I$. Let $f:K=\prod_{i\in I} K_i\to \calM^\sa$ be an nc function. We say $f$ is \emph{separately nc convex} if whenever $x,y\in K$, and $x\prec_\alpha y$ is a dilation in $K$ such that at most one of the dilations
\[
x_i\prec_\alpha y_i, \quad i\in I,
\]
is not trivial, we have
\[
f(x)\le \alpha^\ast f(y)\alpha.
\]
\end{defn}

The following observation justifies the terminology ``separately nc convex".

\begin{prop}\label{prop:separately_levelwise}
Let $f:K=\prod_{i\in I} K_i\to \calM^\sa$ be an nc function. Then $f$ is separately nc convex if and only the restriction $f\vert_{K_n}$ is a separately convex function for each matrix level $K(n)$ of $K$.
\end{prop}
\begin{proof}
The proof is essentially the same as for \cite[Proposition 7.2.3]{davidson_noncommutative_2019}, so we only provide a sketch in the two-variable case $I=\{1,2\}$.

Suppose $f:K_1\times K_2\to \calM^\sa$ is separately nc convex. Given $x,y\in K_1(n)$, $\lambda \in [0,1]$, and $z\in K_2(n)$, the dilation
\[
((1-\lambda)x+\lambda y,z) =
\begin{pmatrix}
    \sqrt{1-\lambda} & \sqrt{\lambda}
\end{pmatrix}\left(\begin{pmatrix}
    x & 0 \\
    0 & y
\end{pmatrix},\begin{pmatrix}
    z & 0 \\
    0 & z
\end{pmatrix}\right)\begin{pmatrix}
    \sqrt{1-\lambda} \\ \sqrt{\lambda}
\end{pmatrix}
\]
is trivial in \emph{every} entry. So, separate nc convexity gives
\begin{align*}
f((1-\lambda)x+\lambda y,z)&\le
\begin{pmatrix}
    \sqrt{1-\lambda} & \sqrt{\lambda}
\end{pmatrix}\begin{pmatrix}
    f(x,z) & 0 \\
    0 & f(y,z)
\end{pmatrix}\begin{pmatrix}
    \sqrt{1-\lambda} \\ \sqrt{\lambda}
\end{pmatrix} \\ &=
(1-\lambda)f(x,z)+\lambda f(y,z).
\end{align*}
Thus $f$ is convex in its first argument, and a symmetrical argument works for the second argument

Conversely, suppose $f$ is separately convex at each level. Take any dilation of the form
\[
(x,z)\prec_\alpha (y,w)
\]
in $K_1\times K_2$ where $w\cong z\oplus z'$ with respect to $\ran(\alpha)$. Up to a unitary, which the nc function $f$ respects, we may write
\[
y=
\begin{pmatrix}
    x & \ast \\
    \ast & \ast
\end{pmatrix},\quad
w=
\begin{pmatrix}
    z & 0 \\
    0 & z'
\end{pmatrix},\quad
f(y,w) =
\begin{pmatrix}
    a & \ast \\
    \ast & \ast
\end{pmatrix},\quad\text{and}\quad
\alpha = \begin{pmatrix}
    I \\ 0
\end{pmatrix}
\]
as operator matrices. Define a selfadjoint unitary
\[
U:=
\alpha\alpha^\ast - (I-\alpha\alpha^\ast) =
\begin{pmatrix}
    I & 0 \\
    0 & -I
\end{pmatrix}.
\]
Then by assumption
\[
f\left(\frac{y+UyU}{2},w\right) =
\begin{pmatrix}
    f(x,w) & 0 \\
    0 & f(\ast,w)
\end{pmatrix}\le
\frac{f(y,w)+Uf(y,w)U}{2} =
\begin{pmatrix}
    a & 0 \\
    0 & \ast
\end{pmatrix}.
\]
Cutting down to the (1,1) corner shows
\[
f(x,z)\le b=\alpha^\ast f(y,w)\alpha.
\]
Hence $f$ respects dilations which are trivial in the second argument, and a symmetric argument works when the dilation is trivial in its first argument. This shows that $f$ is separately nc convex.
\end{proof}

\begin{eg}\label{eg:nc_xy}
Let $K_1,K_2\subseteq \calM^\sa$ be compact nc convex sets. Proposition \ref{prop:separately_levelwise} implies that the nc function $f:K_1\times K_2\to \calM^\sa$ defined by
\[
f(x,y)=
\frac{xy+yx}{2}
\]
is separately nc convex. It is not an nc convex function, because its restriction to the first level $(K_1\times K_2)_1$ is the non-convex function $f(x,y)=xy$ from Example \ref{eg:classical_separately_convex}.
\end{eg}

\begin{eg}\label{eg:affine_sum_prod}
More generally, Proposition \ref{prop:separately_levelwise} shows that any function of the form
\[
a_1\cdots a_k+a_k^\ast\cdots a_1^\ast \in C(K_1\times\cdots \times K_d)
\]
for nc affine functions $a_i\in A(K_{j_i})$ in separate variables, i.e. with $i\ne i'$ implying $j_i\ne j_{i'}$, is separately nc convex. Moreover, sums of such functions are also separately nc convex.
\end{eg}

\begin{eg}\label{eg:affine_ab2a}
If $I$ is a finite set and $J\subseteq I$, we identify
\[
C\left(\prod_{j\in J}K_j\right)\cong
\ast_{j\in J}C(K_j)\subseteq
\ast_{i\in I}C(K_i)\cong
C\left(\prod_{i\in I} K_i\right)
\]
as a C*-subalgebra. If $f\in C(\prod_{j\in J}K_j)_+$ is a positive separately nc convex nc function, $i\in I\setminus J$, and $a\in A(K_i)\subseteq C(K_i)\subseteq C(\prod_i K_i)$ is an nc affine function, then the function $F:=a^\ast fa\in C(\prod_i K_i)$ is separately nc convex. 

Indeed, suppose $x=(x_i)_{i\in I}\prec_\alpha y=(y_i)_{i\in I}$ is a dilation that is nontrivial in at most one coordinate. Indeed, if $\pi_J:\prod_{i\in I}K_i\to \prod_{j\in J}K_J$ is the natural projection, then
\begin{align*}
F(x) &= \alpha^\ast a(y_i)^\ast f(\alpha^\ast \pi_J(y) \alpha)\alpha^\ast a(y_i)\alpha \\ &\le
\alpha^\ast a(y_i)^\ast (\alpha\alpha^\ast)f(\pi_J(y))(\alpha\alpha^\ast)a(y_i)\alpha.
\end{align*}
If the dilation $x_i\prec_\alpha y_i$ is trivial, then $a(x_i)$ commutes with $\alpha\alpha^\ast$ and the right hand side becomes
\[
\alpha^\ast a(y_i)^\ast f(\pi_J(y))a(y_i)\alpha=
\alpha^\ast F(y)\alpha.
\]
Otherwise, $x_j\prec_\alpha y_j$ for each $j\in J$, so $f(\pi_J(y))$ commutes with $\alpha\alpha^\ast$. Since $f(\pi_J(y))\ge 0$, the right hand side is
\[
\alpha^\ast a(y_i)^\ast f(\pi_J(y))^{1/2}aa^\ast f(\pi_J(y))^{1/2} a(y_i)\alpha\le
\alpha^\ast a(y_i)^\ast f(\pi_J(y))a(y_i)\alpha = \alpha^\ast F(y)\alpha.
\]
In either case, $F(x)\le \alpha^\ast F(y)\alpha$.

For example, if $a\in A(K_i)$ and $b\in A(K_j)$ are nc affine functions in separate variables $i\ne j$, then
\[
F=
a^\ast b^\ast ba
\]
is separately nc convex. An easy induction shows that if $a_1\in A(K_{i_1}),\ldots,a_k\in A(K_{i_k})$ with $i_1,\ldots,i_k$ distinct indices, then the function
\[
F=
a_1^\ast\cdots a_k^\ast a_k\cdots a_1
\]
is separately nc convex.
\end{eg}

The following definition is designed to exactly capture the largest class of ucp maps for which our approach can prove a Jensen-type inequality that characterizes separately nc convex functions. The name and involved chain of dilations are meant as a noncommutative analogue of the role that Fubini's theorem plays in the proof of Proposition \ref{prop:classical_separate_jensen}.

\begin{defn}\label{def:fubini_type}
Let $A_i$, $i\in I$, be unital C*-algebras. A ucp map $\mu:\ast_{i\in I} A_i\to B(H)$ is of \emph{Fubini type} if there exists an ordinal $\alpha$ and a chain of dilations 
\[
\{\mu_\lambda:\ast_{i\in I}A_i\to B(H)\mid \lambda \le \alpha\}
\]
such that
\begin{itemize}
\item [(i)] $\mu_0=\mu$ and $H_0=H$,
\item [(ii)] $\mu_\alpha$ is a $\ast$-homomorphism,
\item [(iii)] $\lambda\le \rho\le \alpha$ implies that $H_\lambda\subseteq H_\rho$ and $\mu_\lambda\prec \mu_\rho$,
\item [(iii)] each dilation
\[
\mu_\lambda \prec \mu_{\lambda +1}
\]
is nontrivial in at most one of the algebras $A_i$, $i\in I$,
\item [(iv)] if $\beta\le \alpha$ is a limit ordinal, then $H_\beta=\overline{\bigcup_{\lambda \le \beta} H_\lambda}$.
\end{itemize}
\end{defn}

In most examples, when $I$ is finite we will take $\alpha$ to be a finite ordinal with $|\alpha|=d:=|I|$. Usually, we can arrange that
\[
\mu=\mu_0\prec \mu_1\prec\cdots \prec \mu_d=\pi
\]
where $\pi$ is a $\ast$-homomorphism and each dilation $\mu_{k-1}\prec \mu_k$ is nontrivial only in the $k$th coordinate.

\begin{eg}\label{eg:almost_star_hom}
Suppose $\mu:\ast_{i\in I}A_i\to B(H)$ is a ucp map such that all but one of the ucp maps $\mu\vert_{A_i}$ is a $\ast$-homomorphism. Then $\mu$ is of Fubini type. Let $\pi:\ast_{i\in I}A_i\to B(K)$ be the minimal Stinespring dilation of $\mu$, with $K\supseteq H$. Then since $\mu=P_H\pi\vert_H$ is a $\ast$-homomorphism on all but one algebra $A_i$, it follows that $H$ is reducing for $\pi(A_i)$ for all but possibly one $i\in I$. Hence the dilation
\[
\mu=:\mu_0\prec \mu_1 := \pi
\]
is trivial in all but one algebra $A_i$.
\end{eg}

In the following theorem, we freely identify $C(\prod_i K_i)$ with $\ast_i C(K_i)$.

\begin{thm}\label{thm:separate_nc_jensen}
Let $K_1,\ldots,K_d$ be compact nc convex sets, and let
\[
f:K_1\times \cdots \times K_d\to \calM^\sa
\]
be a bounded and (weak-$\ast$ to weak-$\ast$ or ultrastrong-$\ast$ to ultrastrong-$\ast$) upper semicontinuous nc function. If $f$ is separately nc convex, and
\[
\mu:\ast_{i\in I}C(K_i)\to M_m
\]
is a ucp map of Fubini type, then the Jensen inequality
\[
f(\bary(\mu))\le
\mu(f)
\]
holds.
\end{thm}

\begin{proof}
Conversely, suppose $f$ is continuous and separately nc convex. Let 
\[
\mu:C(\prod_{i\in I} K_i)\cong \ast_{i\in I}C(K_i)\to M_m=B(H)
\]
be a ucp map of Fubini type, with associated dilation chain $\{\mu_\lambda \mid \lambda \le \alpha\}$, where $\mu_0=\mu$ and $\pi:=\mu_\alpha$ is a $\ast$-homomorphism. Set $x_\lambda:=\bary(\mu_\lambda)$ for all $\lambda\le \alpha$. Because $\pi$ is a $\ast$-homomorphism, we have $\pi=\delta_{x_{\alpha}}$. Since the barycenter map is continuous and nc affine, then $\lambda\le \rho\le \alpha$ implies $x_\lambda\prec x_\rho$, and whenever $\lambda+1\le \alpha$, the dilation $x_\lambda\prec_{\alpha_{\lambda+1}} x_{\lambda+1}$ in $K_1\times\cdots\times K_d$ is trivial in all but at most one variable. That is, the chain of dilations $\{x_\lambda \mid \lambda \le \alpha\}$ shows that $x\in K\cong S(A(K))$ is itself of Fubini type when viewed as a ucp map on $A(K)$.

We will show that
\[
f(x_0)\le P_{H_0} f(x_\lambda)\vert_{H_0}
\]
for any $\lambda$ by transfinite induction on $\lambda$. This is a tautology when $\lambda=0$. Suppose for $\lambda\le \alpha$ that $f(x_0)\le P_{H_0} f(x_\lambda)\vert_{H_0}$. Because the dilation $x_\lambda\prec x_{\lambda+1}$ is trivial in all but one variable, and $f$ is separately nc convex, we have
\[
f(x_\lambda)\le
P_{H_{\lambda}} f(x_{\lambda+1})\vert_{H_{\lambda}}.
\]
Compressing to $H_0$ and using the inductive hypothesis yields 
\[
f(x_0)\le 
P_{H_0}f(x_\lambda)\vert_{H_0}\le P_{H_0}f(x_{\lambda+1})\vert_{H_0}.
\]
Finally, suppose $\beta\le \alpha$ is a limit ordinal, and $f(x_0)\le P_{H_0}f(x_\lambda)\vert_{H_0}$ for all $\lambda<\beta$. Fix any constant vector $c\in (K_1\times \cdots \times K_d)_1$. Define a net
\[
z_\lambda =
x_\lambda \oplus c1_{d_\lambda}
\]
with respect to the decomposition $H_\beta=H_\lambda\oplus (H_\beta\ominus H_\lambda)$, where $d_\lambda=\dim (H_\beta\ominus H_\lambda)$. Then $z_\lambda$ converges to $x_\beta$ ultrastrong-$\ast$. For $\rho\le \lambda<\beta$ we have
\[
P_{H_\rho}f(z_\lambda)\vert_{H_\rho} =
P_{H_\rho}f(x_\lambda)\vert_{H_\rho}
\]
as $f$ respects direct sums. Compressing to $H_0$ gives
\[
P_{H_0}f(z_\lambda)\vert_{H_0}=P_{H_0}f(x_\lambda)\vert_{H_0}\ge f(x_0)
\]
by inductive hypothesis. Since $z_\lambda\to x$ and $f$ is upper semicontinuous, we have
\[
P_{H_0}f(x_\beta)\vert_{H_0}\ge
P_{H_0}(\limsup_{\lambda<\beta} f(z_\lambda))\vert_{H_0}\ge
\limsup_{\lambda<\beta}P_{H_0}f(z_\lambda)\vert_{H_0}\ge f(x_0).
\]
This completes the induction, and by taking $\lambda=\alpha$ we conclude
\[
f(\bary(\mu))=
f(x_0)\le
P_{H_0}f(x_\alpha)\vert_{H_0} =
P_{H_0}\pi(f)\vert_{H_0}=\mu(f).\qedhere.
\]
\end{proof}

\begin{rem}\label{rem:cts_choices}
Theorem \ref{thm:separate_nc_jensen}, applies to weak-$\ast$ or ultrastrong-$\ast$ continuous nc functions, including all functions in $C(K_1\times\cdots \times K_d)^\sa$. The same proof also shows that a non-continuous separately nc convex function still satisfies a Jensen inequality $f(\bary(\mu))\le \alpha^\ast \pi(f)\alpha$ whenever $\mu$ is a Fubini type ucp map with an associated dilation chain of finite length.
\end{rem}

\begin{rem}\label{rem:jensen_characterizes}
In fact, the Jensen inequality in Theorem \ref{thm:separate_nc_jensen} completely characterizes separate nc convexity of $f$. Suppose $f$ satisfies the claimed Jensen inequality for ucp maps of Fubini type. Let
\[
x\prec_\alpha y
\]
be a dilation in $\prod_{i\in I}K_i$ which is trivial in all but at most one coordinate. Let $\mu=\alpha^\ast \delta_y\alpha$. Because the dilation $\mu\prec \delta_y$ is trivial in all but possibly one algebra $C(K_i)$, the ucp map $\mu$ is of Fubini type. Therefore 
\[
f(x)=f(\bary(\mu))\le
\mu(f)=\alpha^\ast f(y)\alpha,
\]
so $f$ is separately nc convex.
\end{rem}

Theorem \ref{thm:separate_nc_jensen} applies to the following wide class of ucp maps which we might consider ``free products" of ucp maps. This definition is just a reorganizing of the construction given by Davidson and Kakariadis \cite{davidson_proof_2019}. Given an index set $I$, let $S_I$ denote the set of finite words in $I$ without repeated letters. We include the empty word $\emptyset$ in $S_I$.

\begin{defn}\label{def:ucp_free_product}
Let $A_i$, $i\in I$, be unital C*-algebras with unital free product $A=\ast_{i\in I}A_i$. Let $\mu:A\to B(H)$ be a ucp map with minimal Stinespring dilation $\pi:A\to B(K)$, where $K\supseteq H$. Define subspaces $H_w$ for $w\in S_I$ by setting
\[
K_\emptyset=H,
\]
and inductively
\[
K_{iw}:=
\overline{\pi(A_i)K_w}\ominus K_w
\]
whenever $w=w_1\cdots w_m\in S_I$ with $w_1\ne i\in I$. We call $\mu$ a \emph{free product ucp map} (of the ucp maps $\mu_i=\mu\vert_{A_i}$, $i\in I$) if the spaces $K_w$ are pairwise orthogonal for $w\in S_I$.
\end{defn}

Minimality of the Stinespring dilation in Definition \ref{def:ucp_free_product} implies that
\[
K=\overline{\sum_{w\in S_I}K_w},
\]
so the definition of ``free product" is just that this sum is direct.

\begin{rem}\label{rem:free_product_examples}
By definition, any ucp map $\mu:A\to B(H)$ which is built from ucp maps $\mu_i:A_i\to B(H)$ via the proof of \cite[Theorem 3.1]{davidson_proof_2019} is a free product ucp map. Examining the proof of \cite[Theorem 3.4]{davidson_proof_2019} also shows that any ucp map built from Boca's theorem, which produces a unique product map $\mu$ given the additional data of prescribed states $\varphi_i:A_i\to \C$, is also a free product ucp map.
\end{rem}

\begin{prop}\label{prop:free_product_fubini}
Let $A_i$, $i\in I$ be unital C*-algebras with unital free product $A=\ast_{i\in I}A$. Then any free product ucp map $\mu:A\to B(H)$ is a ucp map of Fubini type.
\end{prop}
\begin{proof}
Suppose $\mu$ is a free product ucp map. Let $\pi:A\to B(K)$ be its minimal Stinespring dilation and define the spaces $K_w$, $w\in S_I$, exactly as in Definition \ref{def:ucp_free_product}, so that
\[
K=
\bigoplus_{w\in S_I} K_w
\]
as an orthogonal direct sum. Up to a fixed bijection we may assume $I$ is just some ordinal $\alpha$. For any ordinal $\lambda\le \alpha$, let
\[
H_\lambda = \bigoplus_{w\in S_\lambda} K_w.
\]
Then $H_0=H$, $H_\alpha=K$, and the sequence $(H_\lambda)_{\lambda \le \alpha}$ is increasing with
\[
H_\beta =
\overline{\bigcup_{\lambda <\beta}H_\lambda}
\]
for any limit ordinal $\beta$.

Set $\mu_\lambda:=P_{H_\lambda}\pi\vert_{H_\lambda}$. It suffices to show that for any $\lambda$ with $\lambda+1\le \alpha$ that the dilation $\mu_\lambda\prec \mu_{\lambda+1}$ is trivial in all variables except $\lambda+1$. Since the sum $K=\bigoplus_w K_w$ is direct, we find
\[
H_{\lambda+1}\ominus H_\lambda = 
\bigoplus\{K_w\mid w=w_1\cdots w_m\in S_{\lambda+1} \text{ with some }w_j=\lambda+1\}.
\]
Given $i\in I$ and $w=w_1\cdots w_m\in S_I$, by definition $\pi(A_i)$ maps $K_w$ into
\[
\begin{cases}
    K_w\oplus K_{iw} & i\ne w_1,\\
    K_w\oplus K_{w_2\cdots w_m} & i= w_1.
\end{cases}
\]
Knowing this, it follows that for any $i\in I$ with $i\ne \lambda+1$ that $\pi(A_i)$ maps $H_{\lambda+1}\ominus H_\lambda$ into
\begin{align*}
\bigoplus\{K_w & \mid w=w_1\cdots w_m \in S_{(\lambda+1)\cup\{i\}} \text{ with some } w_j=\lambda+1\} \\ &\subseteq
\bigoplus\{K_w\mid w=w_1\cdots w_m\in S_I \text{ with some } w_j \ne i\} = (H_\lambda)^\perp.
\end{align*}
Compressing to $H_{\lambda+1}$ shows that $H_{\lambda+1}\ominus H_\lambda$ is invariant for $\mu_{\lambda+1}(A_i)=P_{H_{\lambda+1}}\pi(A_i)\vert_{H_{\lambda+1}}$. So, the dilation $\mu_\lambda\prec \mu_{\lambda+1}$ is trivial in any variable $i\ne \lambda+1$. The chain of dilations $\{\mu_\lambda\mid \lambda \le \alpha\}$ shows that $\mu$ is of Fubini type.
\end{proof}

\begin{eg}\label{eg:non_product_fubini}
Fix some large $M\ge 2$, and let $I=[-M,M]\subseteq \mathbb{R}$ be an interval, so we have a compact nc convex set
\[
\MIN(I)=
\{x\in \calM^\sa\mid \sigma(x)\subseteq I\}.
\]
Set $A_1=A_2=C(\MIN(I))$. By Corollary \ref{cor:C_product} we may identify
\[
A_1\ast A_2\cong C(\MIN(I)\times \MIN(I)).
\]

Let
\[
y=(y_1,y_2):=
\left(
\begin{pmatrix}
    1 & 1 & 0 \\
    1 & 1 & 0 \\
    0 & 0 & 1
\end{pmatrix},
\begin{pmatrix}
    1 & 0 & 1 \\
    0 & 1 & 1 \\
    1 & 1 & 1
\end{pmatrix}
\right)\in (\MIN(I)\times \MIN(I))(3) \subseteq M_3(\C)^2.
\]
Setting $K:=\C^3$, the evaluation map
\[
\pi:=\delta_y:A_1\ast A_2\cong C(\MIN(I)\times \MIN(I))\to M_3=B(K)
\]
is a $\ast$-homomorphism. Define subspaces
\begin{align*}
H=H &:=
\C\oplus 0\oplus 0 \quad\text{and}\\
H_1 &:=
\C\oplus \C\oplus 0.
\end{align*}
Compress to get ucp maps
\begin{align*}
\mu:=\mu_0&:=
P_{H}\pi\vert_{H},\\
\mu_1&:=
P_{H_1}\pi\vert_{H_0}.
\end{align*}
Because $H$ is reducing for $\mu_1(A_2)$, and $H_1$ is reducing for $\pi(A_1)$, the chain of dilations
\[
\mu\prec\mu_1\prec \pi
\]
demonstrates that $\mu$ is a ucp map of Fubini type.

However, $\mu$ is not a free product ucp map. For $I=\{1,2\}$, define the subspaces $K_w\subseteq K$ for $w\in S_I$ as in Definition \ref{def:ucp_free_product}. Because
\[
C(\MIN(I))\cong C(I)
\]
is generated by the polynomials, it is straightforward to see that
\[
\pi(A_1)H =
C^\ast(y_1)H=
\C\oplus \C\oplus 0 =H_1.
\]
Hence
\[
K_1 = \overline{\pi(A_1)H}\ominus H = 
0\oplus \C \oplus 0.
\]
Then, we have
\[
\pi(A_2)K_1 = 
C^\ast(y_2)K_1 =
\C^3=K.
\]
For instance,  if $\{e_1,e_2,e_3\}$ is the standard basis for $\C^3$, then $e_1=(y_1^2-2y_1)e_2$ and $e_3 = (y_1-I)e_2$ both lie in $C^\ast(y_2)K_1$. Thus
\[
K_{21}=
\overline{\pi(A_2)K_1}\ominus K_1 =
\C\oplus 0 \oplus \C,
\]
which is not orthogonal to $H_\emptyset=H=\C\oplus 0 \oplus 0$. Therefore
\[
K=H+K_1+K_{21},
\]
but this sum is not direct, and so $\mu$ isn't a free product ucp map.
\end{eg}

\section{Connection to Free Probability}\label{sec:free_probability}

Given C*-algebras $A_i$, $i\in I$, and prescribed states $\varphi_i:A_i\to \C$, a ucp map
\[
\mu:\ast_{i\in I}A_i\to B(H)
\]
is \emph{conditionally free} (with respect to the family $(\varphi_i)_{i\in I}$) if for every reduced word
\[
a_1\cdots a_m\in \ast_{i\in I}A_i 
\]
(reduced meaning that $a_k\in A_{j(k)}$ with $j(1)\ne \cdots \ne j(m)$) that satisfies
\[
\varphi_{j(k)}(a_k)=0 \text{ for all }k=1,\ldots,m,
\]
the multiplication rule
\[
\mu(a_1\cdots a_m)=
\mu(a_1)\cdots \mu(a_m)
\]
holds \cite{bozejko_convolution_1996,mlotkowski_operator-valued_2002}. Boca's theorem (\cite{boca_free_1991} or see \cite[Theorem 3.4]{davidson_proof_2019}) shows that if $\mu_i:A_i\to B(H)$ are any ucp maps and $\varphi_i\in S(A_i)$ are prescribed states, there exists a unique ucp map
\[
\mu:\ast_{i\in I}A_i\to B(H)
\]
which is conditionally free with respect to $(\varphi_i)_{i\in I}$.

If $\mu$ is any ucp map which is conditionally free with respect to some family of states $(\varphi_i:A_i\to \C)_{i\in I}$, then $\mu$ is a free product ucp map in the sense of Definition \ref{def:ucp_free_product}. This follows from the uniqueness in Boca's theorem. Any such conditionally free map agrees with the one constructed by Boca's theorem, and so Remark \ref{rem:free_product_examples} applies. Or, examining the proof of \cite[Theorem 3.2]{davidson_proof_2019} in the case of amalgamation over $\mathbb{C}$ shows how to build the Stinespring dilation. The minimal Stinespring dilation
\[
\pi:\ast_{i\in I}A_i\to B(K)
\]
lives on a direct sum
\[
K=
\bigoplus_{w \in S_I} K_w,
\]
where $S_I$ is the set of words with letters in $I$ without repeated letters. The dilation is constructed recursively so that whenever $i\in I$ and $w=w_1\cdots w_m\in S_I$ with $i\ne w_1$, the sum $H_w\oplus H_{iw}$ is reducing for $\pi(A_i)$, and the compression $P_{H_w}\pi\vert_{H_w}$ to $H_w$ satisfies
\[
P_{H_w}\pi\vert_{H_w}=
\varphi_i\otimes \id_{H_w}.
\]
Applying Theorem \ref{thm:separate_nc_jensen} in this context yields the following.

\begin{cor}\label{cor:free_prob_jensen}
Suppose $A_1,\ldots,A_d$ are unital C*-algebras such that $A_i=\Cmax(S_i)$ for some operator systems $S_1=A(K_1),\ldots,S_d=A(K_d)$, where $K_i\cong \calS(S_i)$ are compact nc convex sets. If
\[
\mu:A_1\ast\cdots\ast A_d\to B(H)
\]
is a conditionally free ucp map, and
\[
f\in (A_1\ast\cdots \ast A_d)^\sa \cong
C(K_1\times\cdots \times K_d)^\sa
\]
is a continuous separately nc convex function, then the Jensen inequality
\[
f(\bary(\mu))\le \mu(f)
\]
holds.
\end{cor}

As in Remark \ref{eg:affine_sum_prod}, Corollary \ref{cor:free_prob_jensen} applies e.g. to any function $f$ which is a (limit of) sum(s) of the form
\[
a_1\cdots a_k + a_k^\ast \cdots a_1^\ast,
\]
where each $a_i\in A(K_{j_i})\cong S_i$ is continuous and nc affine, and no two $a_i$'s depend on the same variable, i.e. $i\ne i'$ implies $j_i\ne j_{i'}$.

In the commutative case, if $K_i\subseteq \mathbb{R}^{m_i}$ are (classical) Choquet simplices, then
\[
C(\MIN(K_i))\cong
C(K_i)
\]
is a commutative C*-algebra, where the isomorphism is implemented by restriction to the first level. In particular, the theorem applies to commutative C*-algebras of the form $A_i=C(I_i)$, where $I_i\subseteq \R$ are intervals. Any selfadjoint element $a_i$ with spectrum $\sigma(a_i)=I_i$ generates such a C*-algebra. What follows is an application of Corollary \ref{cor:free_prob_jensen} in the special case where $a_i$ are free semicircular elements.

\begin{eg}\label{eg:semicircular_jensen}
Suppose $(A,\varphi)$ is a C*-probability space, where $A$ is a unital C*-algebra, and $\varphi$ is a faithful tracial state on $A$. Suppose $a,b\in A^\sa$ are free semicircular elements. This means that $a$ and $b$ have the semicircular $\ast$-distribution with some radii $r>0$ and $s>0$, respectively, and that the C*-algebras $C^\ast(A)$ and $C^\ast(b)$ are freely independent with respect to $\varphi$. Then $\sigma(a)=I:=[-r,r]$ and $\sigma(b)=J=[-s,s]$. By \cite[Theorem 7.9]{nica_lectures_2006}, faithfulness and traciality of $\varphi$ implies that
\[
C^\ast(a,b)\cong
C^\ast(a)\ast C^\ast(b)
\]
via the natural map.

Consider the operator systems $S=\spn\{1_A,a\}$ and $T=\spn\{1_A,b\}$ generated by $a$ and by $b$. Because $\sigma(a)=I$ and $\sigma(b)=J$ are closed intervals, the functional calculus gives standard isomorphisms
\[
C^\ast(a)\cong
C(I)\cong
C(\MIN(I))\quad\text{and}\quad
C^\ast(b)\cong
C(J)\cong
C(\MIN(J)).
\]
We then identify
\[
C^\ast(a,b)\cong
C(\MIN(I))\ast C(\MIN(J))\cong
C(\MIN(I)\times \MIN(J)).
\]
Suppose
\[
\mu:C^\ast(a)\ast C^\ast(b)\cong C(\MIN(I)\times \MIN(J))\to B(H)
\]
is a conditionally free ucp map, or a ucp map built from Boca's theorem. The barycenter $\bary(\mu)\in \MIN(I)\times \MIN(J)$ corresponds to the point evaluation
\[
\sigma=
\delta_{\bary(\mu)}:C^\ast(a)\ast C^\ast(b)\to B(H)
\]
that is the unique $\ast$-homomorphism determined by $\sigma(a)=\mu(a)$ and $\sigma(b)=\mu(b)$. Note that such a $\ast$-homomorphism exists because $\sigma(\mu(a))\subseteq I$ and $\sigma(\mu(b))\subseteq J$.

Theorem \ref{thm:separate_nc_jensen} implies that for every element $x\in C^\ast(a)\ast C^\ast(b)\cong C(\MIN(I)\times \MIN(J))$ which corresponds to a separately nc convex function on $\MIN(I)\times \MIN(J)$, the operator inequality $\sigma(x)\le \mu(x)$ holds. By Examples \ref{eg:affine_sum_prod} and \ref{eg:affine_ab2a}, elements of the form $x=ab+ba$, or $x=ab^2 a$ correspond to separately convex functions. Therefore if $a$ and $b$ are free semicircular elements and $\mu$ is a conditionally free ucp map, we get the inequalities
\begin{align}
\mu(a)\mu(b)+\mu(b)\mu(a)&\le
\mu(ab+ba),\label{eq:semicircular_inequalities_1}\\
\mu(a)\mu(b)^2\mu(a)&\le 
\mu(ab^2a) \label{eq:semicircular_inequalities_2}.
\end{align}
\end{eg}

In this same context, the ``one-variable" nc Jensen inequality of Davidson and Kennedy \cite[Theorem 7.6.1]{davidson_noncommutative_2019} implies that if $y\in S+T=\spn\{1_A,a,b\}$, we have
\[
\mu(y)^\ast\mu(y)\le\mu(y^\ast y)
\]
because the element $x=y^\ast y$ corresponds to a (jointly) nc convex function. In this case, such an inequality trivially reduces to the usual Schwarz inequality for ucp maps. In contrast, the inequalities \eqref{eq:semicircular_inequalities_1} and \eqref{eq:semicircular_inequalities_2} do not reduce to some trivial application of the Schwarz inequality because they do not hold for general ucp maps $\mu$. For instance, one could take $A=M_2$, 
\[
a=\begin{pmatrix}
    1 & 0 \\
    0 & 0
\end{pmatrix},\qquad
b=\begin{pmatrix}
    0 & 0 \\
    0 & 1
\end{pmatrix},
\]
and let $\mu:M_2\to \C$ be the normalized trace. (In this example, $(A,\varphi):=(M_2,\tr/2)$ is a faithful tracial C*-probability space, but $a$ and $b$ are not freely independent and not semicircular.)

The reasoning of Example \ref{eg:semicircular_jensen} generalizes readily to free semicircular families of arbitrary size as follows.

\begin{cor}\label{cor:semicircular_inequalities}
Let $(A,\varphi)$ be a C*-probability space with faithful tracial state $\varphi\in \calS(A)$. Suppose $a_1,\ldots,a_d\in A$ is a free family of semicircular elements. If
\[
\mu:C^\ast(a_1,\ldots,a_d)\cong
C^\ast(a_1)\ast\cdots\ast C^\ast(a_d)\to B(H)
\]
is a conditionally free ucp map, or a ucp map built from Boca's theorem, then for every list of distinct indices $i_1,\ldots,i_k\in \{1,\ldots,d\}$, the operator inequalities
\[
\mu(a_{i_1})\cdots \mu(a_{i_k})+\mu(a_{i_k})\cdots \mu(a_{i_1})\le
\mu(a_{i_1}\cdots a_{i_k}+a_{i_k}\cdots a_{i_1})
\]
and
\[
\mu(a_{i_1})\cdots \mu(a_{i_{k-1}})\mu(a_{i_k})^2\mu(a_{i_{k-1}})\cdots \mu(a_{i_1})\le
\mu(a_{i_1}\cdots a_{i_k}^2\cdots a_{i_1})
\]
hold.
\end{cor}

Note that in Corollary \ref{cor:semicircular_inequalities}, the ucp map $\mu$ need not be conditionally free with respect to the states $\varphi_i:=\varphi\vert_{C^\ast(a_i)}$. Conditional freeness with respect to \emph{any} family of states is enough.


\section*{Funding} This work was supported by an NSERC Alexander Graham Bell Canada Graduate Scholarship-Doctoral (CGS-D) [grant number 401230185].

\section*{Acknowledgements} The author would like to thank Kenneth Davidson and Matthew Kennedy for many helpful discussions. The author would also like to thank Paul Skoufranis for helpful discussions regarding the connections of this work to free probability.

\bibliographystyle{plain}
\bibliography{nc_jensen_sources_abbrev}

\end{document}